\newif\ifHYPER\global\HYPERtrue
\newif\ifJOURNAL\global\JOURNALfalse
\documentclass[10pt,b5paper,fleqn]{article}
\usepackage[utf8]{inputenc}
\usepackage{amsmath,amsfonts,amssymb}
\usepackage[mathscr]{eucal}
\usepackage{color}
\usepackage{graphicx}
\usepackage{amsthm}

\ifHYPER
\definecolor{ks-green}{rgb}{0.0,0.7,0.0}
\definecolor{ks-red}{rgb}{0.7,0.0,0.0}
\definecolor{ks-blue}{rgb}{0.0,0.0,0.7}
\usepackage{hyperref}
\hypersetup{colorlinks=true,citecolor=ks-green,linkcolor=ks-red}
\fi

\numberwithin{equation}{section}
\theoremstyle{plain}
\newtheorem{theorem}[equation]{Theorem}
\newtheorem{lemma}[equation]{Lemma}
\newtheorem{corollary}[equation]{Corollary}
\newtheorem{proposition}[equation]{Proposition}

\theoremstyle{definition}
\newtheorem{definition}[equation]{Definition}
\newtheorem{Example}[equation]{Example}
\newtheorem{Remark}[equation]{Remark}
\theoremstyle{remark}

\newenvironment{remark}{\emph{Remark.}}{}
\newenvironment{notation}{\emph{Notation.}}{}

\def\trp{^{\!\top}}

\def\inv{^{-1}}
\def\quinv{^+}

\def\numN{\mathbb{N}}

\def\revddots{\mathinner{\mkern1mu\raise1pt\vbox{\kern7pt\hbox{.}}\mkern2mu
  \raise4pt\hbox{.}\mkern2mu\raise7pt\hbox{.}\mkern1mu}}

\newcommand{\ncRATS}[2]{#1^{\text{rat}}\langle\!\langle #2\rangle\!\rangle}

\newcommand{\ncPOWS}[2]{#1\langle\!\langle #2\rangle\!\rangle}
\newcommand{\freeALG}[2]{#1\langle #2\rangle}
\newcommand{\freeFLD}[2]{#1(\!\langle #2\rangle\!)}
\newcommand{\perm}{\Sigma}

\DeclareMathOperator{\rank}{rank}
\DeclareMathOperator{\size}{size}

\DeclareMathOperator{\linsp}{span}

\newcommand{\field}[1]{\mathbb{#1}}
\newcommand{\als}[1]{\mathcal{#1}}

\newcommand{\qdet}[1]{\vert #1 \vert}

\begin{document}
\title{Linearizing the Word Problem\\ in (some) Free Fields}
\author{Konrad Schrempf%
  \footnote{%
    Contact: math@versibilitas.at,
    Universität Wien, Fakultät für Mathematik,
    Oskar-Morgenstern-Platz~1, 1090 Wien, Austria.
    Supported by the Austrian FWF Project P25510-N26
``Spectra on Lamplighter groups and Free Probability''}
  }

\maketitle

\begin{abstract}
We describe a solution of the word problem
in free fields (coming from non-commutative polynomials
over a commutative field) using elementary linear algebra,
provided that the elements are given by
minimal linear representations.
It relies on the normal form of Cohn and Reutenauer
and can be used more generally to (positively) test rational identities.
Moreover we provide a construction of minimal linear representations
for the inverse of non-zero elements.
\end{abstract}



\section*{Introduction}

Free (skew) fields arise as universal objects when it comes to embed the ring of
non-commutative polynomials, that is, polynomials in
(a finite number of) non-commuting variables, into a
skew field
\cite[Chapter~7]{Cohn1985a}
. The notion of ``free fields'' goes back to Amitsur
\cite{Amitsur1966a}
. A brief introduction can be found in
\cite[Section~9.3]{Cohn2003b}
, for details we refer to
\cite[Section~6.4]{Cohn1995a}
.
In the present paper we restrict the setting
to \emph{commutative} ground fields, as a special case.
See also
\cite{Roberts1984a}
.
In
\cite{Cohn1994a}
, Cohn and Reutenauer introduced a \emph{normal form} for
elements in free fields in order to extend results from
the theory of formal languages.
In particular they characterize minimality
of linear representations in terms of linear independence of the entries of
a column and a row vector, generalizing the concept of
``controllability and observability matrix'' (Section~\ref{sec:wp.reg}).

It is difficult to solve the word problem,
that is, given linear representations (LRs for short) 
of two elements $f$ and $g$,
to decide whether $f=g$.
In \cite{Cohn1999a}
\ the authors describe an answer of this question
(for free fields coming from non-commutative polynomials
over a \emph{commutative} field). In practice however,
this technique (using Gröbner bases) can be impractical
even for representations of small dimensions.

Fortunately, it turns out that the word problem is
equivalent to the solvability of a \emph{linear} system of equations
if \emph{both} elements are given by \emph{minimal} linear representations.
Constructions of the latter are known for \emph{regular} elements
(non-commutative rational series),
but in general non-linear techniques are necessary.
This is considered in future work.
Here we present a simple construction of minimal LRs
for the inverses of arbitrary non-zero elements given by
\emph{minimal} LRs.
In particular this applies to the inverses
of non-zero polynomials
with vanishing constant coefficient
(which are not regular anymore).
This is of interest especially for those polynomials
which are \emph{identities} 
\cite{Amitsur1950a}
, for example $xy-yx$ (which vanishes identically on
\emph{commutative} rings).

In any case, \emph{positive} testing of rational identities
becomes easy. Furthermore, the implementation in
computer (algebra) software needs only a basic data structure
for matrices (linear matrix pencil) and an \emph{exact}
solver for linear systems.

\medskip
Section~\ref{sec:wp.rep} introduces the required notation
concerning linear representations and admissible linear
systems in free fields. Rational operations on representation
level are formulated and the related concepts of linearization
and realization are briefly discussed.
Section~\ref{sec:wp.wp} describes the word problem.
Theorem~\ref{thr:wp.wp} shows that the (in general non-linear)
problem of finding appropriate
transformation matrices can be reduced to a \emph{linear system of equations}
if the given LRs are \emph{minimal}.
Examples can be constructed for regular elements (rational series)
as special cases (of elements in the free field), which
are summarized in Section~\ref{sec:wp.reg}.
Here algorithms for obtaining minimal LRs are already known.
Section~\ref{sec:wp.min} provides a first step in
the construction of minimal LRs (with linear techniques),
namely for the inverses of non-zero elements
given itself by a \emph{minimal} linear representations.
This is formulated in Theorem~\ref{thr:wp.mininv}.

\medskip
The main result is Theorem~\ref{thr:wp.wp},
the ``linear'' word problem. Although it is rather elementary,
it opens the possibility to work \emph{directly} on linear
representations (instead of the spaces they ``span'').
Or, using Bergman's words
\cite{Bergman1978a}
: ``The main results in this paper are trivial. But what is trivial
when described in the abstract can be far from clear in the
context of a complicated situation where it is needed.''

\section{Representing Elements}\label{sec:wp.rep}

Although there are several ways for representing elements
in (a subset of) the free field (linear representation
\cite{Cohn1999a}
, linearization
\cite{Cohn1985a}
,
realization
\cite{Helton2006a}
, proper linear system
\cite{Salomaa1978a}
, etc.)
the concept of 
a \emph{linear representation}
seems to be the most convenient.
It has the advantage (among others) that
in the special case of regular elements,
the general definition of the rank coincides with
the Hankel rank
\cite{Fliess1974a}
,
\cite[Section~II.3]{Salomaa1978a}
.\index{Hankel rank}

Closely related to LRs are
\emph{admissible linear systems} (ALS for short)
\cite{Cohn1985a}
, which could be seen as a special case.
Both notations will be used synonymously.
Depending on the context an ALS will be written
as a triple, for example $\als{A} = (u,A,v)$
or as linear system $As = v$,
sometimes as $u = tA$.
Like the rational operations defined on linear
representation level 
\cite{Cohn1999a}
, similar constructions can be done easily
on ALS level. Thus, starting from systems for monomials
(Proposition~\ref{pro:min.mon})
only, a representation for \emph{each} element in the free
field can be constructed recursively.

\begin{notation}
Zero entries in matrices are usually replaced by (lower) dots
to stress the structure of the non-zero entries
unless they result from transformations where there
were possibly non-zero entries before.
We denote by $I_n$ the identity matrix and
$\perm_n$ the permutation matrix that reverses the order
of rows/columns of size $n$.
If the size is clear from the context, $I$ and $\perm$ are used
respectively.
\end{notation}

\medskip
Let $\field{K}$ be a \emph{commutative} field and
$X = \{ x_1, x_2, \ldots, x_d\}$ be a \emph{finite} alphabet.
$\freeALG{\field{K}}{X}$ denotes the \emph{free associative
algebra}\index{free associative algebra}
(or ``algebra of non-commutative polynomials'')
and $\freeFLD{\field{K}}{X}$ denotes the \emph{universal field of
fractions}\index{universal field of fractions}
(or ``free field'') of $\freeALG{\field{K}}{X}$
\cite{Cohn1995a}
,
\cite{Cohn1999a}
. In the examples the alphabet is usually $X=\{x,y,z\}$.

\begin{definition}[Inner Rank, Full Matrix, Hollow Matrix
\cite{Cohn1985a}
, \cite{Cohn1999a}
]
\index{inner rank}\index{full matrix}\index{hollow matrix}%
Given a matrix $A \in \freeALG{\field{K}}{X}^{n \times n}$, the \emph{inner rank}
of $A$ is the smallest number $m\in \numN$
such that there exists a factorization
$A = T U$ with $T \in \freeALG{\field{K}}{X}^{n \times m}$ and
$U \in \freeALG{\field{K}}{X}^{m \times n}$.
The matrix $A$ is called \emph{full} if $m = n$,
\emph{non-full} otherwise. It is called \emph{hollow} if
it contains a zero submatrix of
size $k \times l$ with $k + l > n$.
\end{definition}

\begin{definition}[Associated and Stably Associated Matrices
\cite{Cohn1995a}
]\label{def:wp.ass}
\index{associated matrix}\index{stable associated matrix}%
Two matrices $A$ and $B$ over $\freeALG{\field{K}}{X}$ (of the same size)
are called \emph{associated} over a subring $R\subseteq \freeALG{\field{K}}{X}$ 
if there exist
invertible matrices $P,Q$ over $R$ such that
$A = P B Q$. $A$ and $B$ (not necessarily of the same size)
are called \emph{stably associated}
if $A\oplus I_p$ and $B\oplus I_q$ are associated for some unit
matrices $I_p$ and $I_q$.
Here by $C \oplus D$ we denote the diagonal sum
$\bigl[\begin{smallmatrix} C & . \\ . & D \end{smallmatrix}\bigr]$.
\end{definition}

In general it is hard to decide whether a matrix is full or not.
For a linear matrix, that is, a matrix of the form
$A = A_0 \otimes 1 + A_1 \otimes x_1 + \ldots + A_d \otimes x_d$
with $A_\ell$ over $\field{K}$, the following criterion is known,
which is used in (the proof of) Theorem~\ref{thr:wp.cohn99.41}.
If a matrix over $\freeALG{\field{K}}{X}$ is not linear,
then Higman's trick
\cite[Section~5.8]{Cohn1985a}
\ can be used to linearize it by enlargement.\index{linearization by enlargement}
The inner rank is also discussed in
\cite{Fortin2004a}
.

\begin{lemma}[%
\protect{\cite[Corollary~6.3.6]{Cohn1995a}
}]\label{lem:cohn95.636}
A linear square matrix over $\freeALG{\field{K}}{X}$
which is not full is associated over $\field{K}$ to a linear
hollow matrix.
\end{lemma}

\begin{definition}[Linear Representations
\cite{Cohn1994a}
]\label{def:wp.rep}
\index{linear representation}%
Let $f \in \freeFLD{\field{K}}{X}$.
A \emph{linear representation} of $f$ is a triple $(u,A,v)$ with
$u \in \field{K}^{1 \times n}$, $A = A_0 \otimes 1 + A_1 \otimes x_1 + \ldots
+ A_d \otimes x_d$, $A_\ell \in \field{K}^{n\times n}$ and
$v \in \field{K}^{n\times 1}$ such that $A$ is full,
that is,
$A$ is invertible over the free field $\freeFLD{\field{K}}{X}$,
and $f = u A\inv v$.
The \emph{dimension} of the representation is $\dim \, (u,A,v) = n$.
It is called \emph{minimal} if $A$ has the smallest possible dimension
among all linear representations of $f$.
The ``empty'' representation $\pi = (,,)$ is the minimal representation
for $0 \in \freeFLD{\field{K}}{X}$ with $\dim \pi = 0$.
\end{definition}

\begin{remark}
In Definition~\ref{def:wp.lin} it can be seen that $f=u A\inv v$ is
(up to sign) the \emph{Schur complement} of
the linearization
$\bigl[\begin{smallmatrix} 0 & u \\ v & A \end{smallmatrix}\bigr]$
with respect to the upper left $1 \times 1$ block.
\end{remark}

\begin{definition}[\cite{Cohn1999a}
]
\index{equivalent linear representations}%
Two linear representations are called \emph{equivalent} if
they represent the same element.
\end{definition}

\begin{definition}[Rank
\cite{Cohn1999a}
]
\index{rank}%
Let $f \in \freeFLD{\field{K}}{X}$ and $\pi$ be a \emph{minimal}
representation of $f$.
Then the \emph{rank} of $f$ is
defined as $\rank f = \dim \pi$.
\end{definition}

\begin{remark}
The connection to the related concepts of \emph{inversion height}
and \emph{depth} can be found in 
\cite{Reutenauer1996a}
, namely \emph{inversion height} $\le$ \emph{depth} $\le$ \emph{rank}.
Additional discussion about the depth appears in
\cite[Section 7.7]{Cohn2006a}
.
\end{remark}

\begin{definition}\label{def:wp.reg}
Let $M = M_1 \otimes x_1 + \ldots + M_d \otimes x_d$.
An element in $\freeFLD{\field{K}}{X}$ is called \emph{regular},
if it has a linear representation $(u,A,v)$ with $A = I - M$,
that is, $A_0 = I$ in Definition~\ref{def:wp.rep},
or equivalently, if $A_0$ is regular (invertible).
\end{definition}

\begin{definition}[Left and Right Families
\cite{Cohn1994a}
]\label{def:cohn94.family}
\index{left family}\index{right family}%
Let $\pi=(u,A,v)$ be a linear representation of $f \in \freeFLD{\field{K}}{X}$
of dimension $n$.
The families $( s_1, s_2, \ldots, s_n )\subseteq \freeFLD{\field{K}}{X}$
with $s_i = (A\inv v)_i$
and $( t_1, t_2, \ldots, t_n )\subseteq \freeFLD{\field{K}}{X}$
with $t_j = (u A\inv)_j$
are called \emph{left family} and \emph{right family} respectively.
$L(\pi) = \linsp \{ s_1, s_2, \ldots, s_n \}$ and
$R(\pi) = \linsp \{ t_1, t_2, \ldots, t_n \}$
denote their linear spans.
\end{definition}

\begin{remark}
The left family $(A\inv v)_i$ (respectively the right family $(u A\inv)_j$)
and the solution vector $s$ of $As = v$ (respectively $t$ of $u = tA$)
will be used synonymously.
\end{remark}

\begin{proposition}[%
\cite{Cohn1994a}
, Proposition 4.7]
A representation $\pi=(u,A,v)$ of an element $f \in \freeFLD{\field{K}}{X}$
is minimal if and only if both, the left family
and the right family, are $\field{K}$-linearly independent.
\label{pro:cohn94.47}
\end{proposition}

\begin{definition}[Admissible Linear Systems
\cite{Cohn1972a}
]\label{def:wp.als}
\index{admissible linear system}\index{ALS|see{admissible linear system}}%
A linear representation $\als{A} = (u,A,v)$ of $f \in \freeFLD{\field{K}}{X}$
is called \emph{admissible linear system} (for $f$),
denoted by $A s = v$,
if $u=e_1=[1,0,\ldots,0]$. The element $f$ is then the first component
of the (unique) solution vector $s$.
\end{definition}

\begin{remark}
In 
\cite{Cohn1985a}
, Cohn defines admissible linear systems with
$v = v_0 \otimes 1 + v_1 \otimes x_1 + \ldots + v_d \otimes x_d$
with $v_i \in \field{K}^{n \times 1}$,
and $u = [0,\ldots,0,1]$. Writing $B = [-v,A]$ as block
of size $n \times (n+1)$ the first $n$ columns of $B$ serve
as numerator and the last $n$ columns of $B$ as denominator.
However, in this setting, for regular elements, the dimension of such a minimal
system could differ from the Hankel rank
\cite{Fliess1974a}
, \cite[Section~II.3]{Salomaa1978a}
.\index{Hankel rank}
\end{remark}

\begin{definition}[Admissible Transformations]
\label{def:trn.adm}
\index{admissible transformation}%
Given a linear representation $\als{A} = (u,A,v)$
of dimension $n$ of $f \in \freeFLD{\field{K}}{X}$
and invertible matrices $P,Q \in \field{K}^{n\times n}$,
the transformed $P\als{A}Q = (uQ, PAQ, Pv)$ is
again a linear representation (of $f$).
If $\als{A}$ is an ALS,
the transformation $(P,Q)$ is called
\emph{admissible} if the first row of $Q$ is $e_1 = [1,0,\ldots,0]$.
\end{definition}

\begin{Remark}[Elementary Transformations]\label{rem:trn.ele}
In practice, transformations can be done by elementary row- and
column operations (with respect to the system matrix $A$).
If we add $\alpha$-times row~$i$ to row~$j\neq i$ in $A$, we
also have to do this in $v$. If we add $\beta$-times column~$i$ to
column~$j\neq i$ we have to \emph{subtract} $\beta$-times row~$j$
from row~$i$ in $s$. Since it is not allowed to change the
first entry of $s$,
column~1 cannot be used to eliminate entries in other columns!
As an example consider the ALS
\begin{displaymath}
\begin{bmatrix}
1 & x-1 \\
. & 1
\end{bmatrix}
s =
\begin{bmatrix}
. \\ 1
\end{bmatrix},\quad
s =
\begin{bmatrix}
1-x \\
1
\end{bmatrix}
\end{displaymath}
for the element $1-x \in \freeFLD{\field{K}}{X}$. Adding column~1
to column~2, that is,
$Q = \bigl[\begin{smallmatrix} 1 & 1 \\ . & 1 \end{smallmatrix}\bigr]$
(and $P = I$), yields the ALS
\begin{displaymath}
\begin{bmatrix}
1 & x \\
. & 1
\end{bmatrix}
s =
\begin{bmatrix}
. \\ 1
\end{bmatrix},\quad
s =
\begin{bmatrix}
-x \\
1
\end{bmatrix}
\end{displaymath}
for the element $-x \neq 1-x$.
\end{Remark}

\begin{proposition}[Rational Operations
\protect{\cite[Section~1]{Cohn1999a}
}]\label{pro:wp.ratop}
Let $f,g,h \in \freeFLD{\field{K}}{X}$ be given by the
admissible linear systems $\als{A}_f = (u_f, A_f, v_f)$,
$\als{A}_g = (u_g, A_g, v_g)$ and
$\als{A}_h = (u_h, A_h, v_h)$
respectively, with $h \ne 0$
and let $\mu \in \field{K}$.
Then admissible linear systems for the rational operations
can be obtained as follows:

\smallskip\noindent
The scalar multiplication
$\mu f$ is given by
\begin{displaymath}
\mu \als{A}_f =
\bigl( u_f, A_f, \mu v_f \bigr).
\end{displaymath}
The sum $f + g$ is given by
\begin{displaymath}
\als{A}_f + \als{A}_g =
\left(
\begin{bmatrix}
u_f & . 
\end{bmatrix},
\begin{bmatrix}
A_f & -A_f u_f\trp u_g \\
. & A_g
\end{bmatrix}, 
\begin{bmatrix} v_f \\ v_g \end{bmatrix}
\right).
\end{displaymath}
The product $fg$ is given by
\begin{displaymath}
\als{A}_f \cdot \als{A}_g =
\left(
\begin{bmatrix}
u_f & . 
\end{bmatrix},
\begin{bmatrix}
A_f & -v_f u_g \\
. & A_g
\end{bmatrix},
\begin{bmatrix}
. \\ v_g
\end{bmatrix}
\right).
\end{displaymath}
And the inverse $h\inv$ is given by
\begin{displaymath}
\als{A}_h\inv =
\left(
\begin{bmatrix}
1 & . 
\end{bmatrix},
\begin{bmatrix}
-v_h & A_h \\
. & u_h
\end{bmatrix},
\begin{bmatrix}
. \\ 1
\end{bmatrix}
\right).
\end{displaymath}
\end{proposition}

\medskip
\begin{remark}
One can easily verify that 
the solution vectors for the admissible linear systems defined above are
\begin{displaymath}
\mu s_f, \quad
\begin{bmatrix} s_f + u_f\trp g \\ s_g \end{bmatrix}, \quad
\begin{bmatrix} s_f g \\ s_g \end{bmatrix} \quad\text{and}\quad
\begin{bmatrix} h\inv \\ s_h h\inv \end{bmatrix}
\end{displaymath}
respectively.
It remains to check that the system matrices are full.
For the sum and the product this is clear from the fact
that the free associative algebra ---being a \emph{free ideal ring} (FIR)---
has \emph{unbounded
generating number} (UGN) and therefore the diagonal
sum of full matrices is full, see
\cite[Section~7.3]{Cohn1985a}
.
The system matrix for the inverse is full because
$h\ne 0$ and therefore the linearization of $\als{A}_h$
is full, compare
\cite[Section~4.5]{Cohn1995a}
.
\end{remark}

\begin{Remark}\label{rem:span.lf}
For the rational operations from Proposition~\ref{pro:wp.ratop} we observe
that the left families satisfy the relations
\begin{align*}
L(\mu \als{A}_f) &= L(\als{A}_f), \\
L(\als{A}_f + \als{A}_g) &= L(\als{A}_f) + L(\als{A}_g)
  \quad\text{and}\quad \\
L(\als{A}_g) &\subseteq L(\als{A}_f \cdot \als{A}_g)
  = L(\als{A}_f) g + L(\als{A}_g).
\end{align*}
And similarly for the right families we have
\begin{align*}
R(\mu \als{A}_f) &= R(\als{A}_f), \\
R(\als{A}_f + \als{A}_g) &= R(\als{A}_f) + R(\als{A}_g)
  \quad\text{and}\quad \\
R(\als{A}_f) &\subseteq R(\als{A}_f \cdot \als{A}_g)
  = R(\als{A}_f) + f R(\als{A}_g).
\end{align*}
\end{Remark}

\begin{definition}[Linearization
\cite{Belinschi2017c}
,
\cite{Cohn1999a}
]\label{def:wp.lin}
\index{linearization}%
Let $f \in \freeFLD{\field{K}}{X}$.
A \emph{linearization} of $f$ is a matrix
$L = L_0 \otimes 1 + L_1 \otimes x_1 + \ldots + L_d \otimes x_d$,
with $L_\ell \in \field{K}^{m \times m}$,
of the form
\begin{displaymath}
L =
\begin{bmatrix}
c & u \\
v & A
\end{bmatrix}
\quad
\in \freeALG{\field{K}}{X}^{m \times m}
\end{displaymath}
such that $A$ is invertible over the free field
and $f$ is the Schur complement, that is, $f = c - u A\inv v$.
If $c=0$ then $L$ is called a \emph{pure} linearization.
The \emph{size} of the linearization is
$\size L = m$, the \emph{dimension} is $\dim L = m-1$.
\end{definition}

\begin{proposition}[\protect{%
\cite[Proposition~3.2]{Belinschi2017c}
}]\label{pro:belinschi17c}
Let $\field{F} = \freeFLD{\field{K}}{X}$ and
$A\in \field{F}^{k \times k}$, $B \in \field{F}^{k \times l}$,
$C \in \field{F}^{l \times k}$ and $D \in \field{F}^{l \times l}$
be given and assume that $D$ is invertible in $\field{F}^{l \times l}$.
Then the matrix
$\bigl[\begin{smallmatrix} A & B \\ C & D \end{smallmatrix}\bigr]$
is invertible in $\field{F}^{(k+l) \times (k+l)}$ if and only
if the \emph{Schur complement} $A-BD\inv C$ is invertible in
$\field{F}^{k \times k}$.
In this case
\begin{displaymath}
\begin{bmatrix}
A & B \\
C & D
\end{bmatrix}
\inv
= 
\begin{bmatrix}
. & . \\
. & D\inv
\end{bmatrix}
+
\begin{bmatrix}
I_k \\ -D\inv C
\end{bmatrix}
\bigl(A - BD\inv C\bigr)\inv
\begin{bmatrix}
I_k & -BD\inv
\end{bmatrix}.
\end{displaymath}
\end{proposition}

\begin{Remark}
(i) Let $f \in \freeFLD{\field{K}}{X}$ be given
by the linearization $L$. Then $f = \qdet{L}_{1,1}$
is the $(1,1)$-quasideterminant
\cite{Gelfand2005a}
\ of $L$.

(ii)
Given a linear representation $(u,A,v)$ of
$f \in \freeFLD{\field{K}}{X}$, then
$L = \bigl[\begin{smallmatrix} . & u \\ -v & A \end{smallmatrix}\bigr]$
is a pure linearization of $f$.

(iii)
Talking about a \emph{minimal} linearization, one has to
specify which class of matrices is considered: Scalar entries
in the first row and column? Pure? And, if applicable,
selfadjoint?
\end{Remark}

\begin{proposition}\label{pro:lin.ext}
Let
\begin{displaymath}
L =
\begin{bmatrix}
 c & u \\ v & A
\end{bmatrix}
\end{displaymath}
be a linearization of size $n$ for some element $f\in\freeFLD{\field{K}}{X}$ and
define another element $g\in \freeFLD{\field{K}}{X}$ by the pure linearization
\begin{displaymath}
\tilde{L} =
\begin{bmatrix} . & \tilde{u} \\ \tilde{v} & \tilde{A} \end{bmatrix}
\quad\text{with}\quad
\tilde{A} =
\begin{bmatrix}
c & u & -1 \\
v & A & . \\
-1 & . & . 
\end{bmatrix},\quad
\tilde{u} = [0,\ldots,0,1],\quad\tilde{v} = \tilde{u}\trp
\end{displaymath}
of size $n+2$.
Then $g = f$.
\end{proposition}

\begin{proof}
Using Proposition~\ref{pro:belinschi17c} ---taking the Schur complement with respect
to the block entry $(2,2)$--- and $b = [-1,0,\ldots,0]$,
the inverse of $\tilde{A}=\bigl[\begin{smallmatrix} L & b\trp \\ b & . \end{smallmatrix}\bigr]$
can be written as
\begin{displaymath}
\tilde{A}\inv =
  \begin{bmatrix} L\inv & . \\ . & . \end{bmatrix}
 -\begin{bmatrix} -L\inv b\trp \\ 1 \end{bmatrix}
  \bigl(b L\inv b\trp\bigr)\inv
  \begin{bmatrix} - b L\inv & 1 \end{bmatrix}.
\end{displaymath}
Hence
\begin{align*}
-\tilde{u}\tilde{A}\inv\tilde{v}
  &= -\Bigl( \begin{bmatrix} . & . \end{bmatrix}
       - (b L\inv b\trp)\inv 
       \begin{bmatrix} -b L\inv & 1 \end{bmatrix} \Bigr)
       \begin{bmatrix} . \\ 1 \end{bmatrix} \\
  &= \bigl(b L\inv b\trp\bigr)\inv \\
  &= \Biggl( b \left( \begin{bmatrix} . & . \\ . & A\inv \end{bmatrix}
               + \begin{bmatrix} 1 \\ -A\inv v \end{bmatrix}
                 \bigl(c - u A\inv v\bigr)\inv
                 \begin{bmatrix} 1 & -u A\inv \end{bmatrix} \right) b\trp \Biggr)\inv \\
  &= \Biggl( \left( \begin{bmatrix} . & . \end{bmatrix}
                 -\bigl(c - u A\inv v\bigr)\inv
                 \begin{bmatrix} 1 & -u A\inv \end{bmatrix} \right)
                 \begin{bmatrix} -1 \\ . \end{bmatrix} \Biggr)\inv \\
  &= c - u A\inv v .\qedhere
\end{align*}
\end{proof}

If the first row or column of a linearization
for some $f \in \freeFLD{\field{K}}{X}$
contains non-scalar entries, then  
Proposition \ref{pro:lin.ext} can be used
to construct a linear representation of $f$.
On the other hand, given a linear representation
of dimension $n$
(of $f$) which can be brought to such a form,
a linearization of size $n-1$ can be obtained.
The characterization of minimality for linearizations
will be considered in future work.

\begin{Example}\label{ex:anticomm}
For the anticommutator $xy + yx$ a minimal ALS is given by
\begin{displaymath}
\left(
\begin{bmatrix}
1 & . & . & .
\end{bmatrix},
\begin{bmatrix}
1 & -x & -y & . \\
. & 1 & . & -y \\
. & . & 1 & -x \\
. & . & . & 1
\end{bmatrix},
\begin{bmatrix}
. \\ . \\ . \\ 1
\end{bmatrix}
\right).
\end{displaymath}
Permuting the columns and multiplying the system matrix by $-1$ we get the linearization
\begin{displaymath}
L'_{xy+yx} =
\begin{bmatrix}
. & . & . & . & 1 \\
. & . & y & x & -1 \\
. & y & . & -1 & . \\
. & x & -1 & . & . \\
1 & -1 & . & . & .
\end{bmatrix}
\end{displaymath}
which is of the form in Proposition \ref{pro:lin.ext} and yields
a \emph{minimal} (pure) linearization of the anticommutator
\begin{displaymath}
L_{xy+yx} =
\begin{bmatrix}
. & y & x  \\
y & . & -1 \\
x & -1 & .
\end{bmatrix}.
\end{displaymath}

\end{Example}

\begin{definition}[Realization
\cite{Helton2006a}
]\label{def:wp.realization}
\index{realization}\index{butterfly realization}%
A \emph{realization} of a matrix 
$F \in \freeFLD{\field{K}}{X}^{p \times q}$
is a quadruple
$(A,B,C,D)$ with $A = A_0 \otimes 1 + A_1 \otimes x_1 + \ldots + A_d \otimes x_d$,
$A_\ell \in \field{K}^{n \times n}$, $B \in \field{K}^{n \times q}$,
$C \in \field{K}^{p \times n}$ and $D \in \field{K}^{p \times q}$
such that $A$ is invertible over the free field and $F = D - C A\inv B$.
The \emph{dimension} of the realization is $\dim\,(A,B,C,D) = n$.
\end{definition}

\begin{remark}
A realization $\mathcal{R} = (A,B,C,D)$ could be written in block form
\begin{displaymath}
L_\mathcal{R} =
\begin{bmatrix}
D & C \\
B & A
\end{bmatrix}\quad
\in \freeALG{\field{K}}{X}^{(p+n) \times (q+n)}.
\end{displaymath}
Here, the definition is such that $F = \qdet{L_\mathcal{R}}_{1',1'}$
is the $(1,1)$-block-quasideterminant
\cite{Gelfand2005a}
\ with respect to block $D$.
For $A = -J+L_A(X)$
we obtain the \emph{descriptor realization} in
\cite{Helton2006a}
.
Realizations where $B$ and/or $C$ contain non-scalar entries
are sometimes called ``butterfly realizations''
\cite{Helton2006a}
.
Minimality with respect to realizations is investigated in
\cite{Volcic2018a}
.
\end{remark}

\section{The Word Problem}\label{sec:wp.wp}

Let $f,g \in \freeFLD{\field{K}}{X}$ be given by the
linear representations $\pi_f = (u_f, A_f, v_f)$
and $\pi_g = (u_g, A_g, v_g)$ of dimension $n_f$ and $n_g$ 
respectively and define the matrix
\begin{displaymath}
L =
\begin{bmatrix}
. & u_f & u_g \\
v_f & A_f & . \\
v_g & . & -A_g
\end{bmatrix},
\end{displaymath}
which is a linearization of $f-g$,
of size $n = n_f + n_g + 1$.
Then $f=g$ if and only if $L$ is not full
\cite[Section~4.5]{Cohn1995a}
.
For the word problem see also 
\cite[Section~6.6]{Cohn1995a}
. Whether $L$ is full or not can be decided
by the following theorem.
For $P = (\alpha_{ij})$ and $Q = (\beta_{ij})$ the
\emph{commutative} polynomial ring is
\begin{align*}
\field{K}[\alpha,\beta] = \field{K}[
  &\alpha_{1,1},\ldots,\alpha_{1,n},
   \alpha_{2,1},\ldots,\alpha_{2,n},\ldots,
   \alpha_{n,1}, \ldots,\alpha_{n,n}, \\
  &\beta_{1,1},\ldots,\beta_{1,n},
   \beta_{2,1},\ldots,\beta_{2,n},\ldots,
   \beta_{n,1},\ldots,\beta_{n,n} ].
\end{align*}

\begin{theorem}[%
\protect{\cite[Theorem~4.1]{Cohn1999a}
}]\label{thr:wp.cohn99.41}
For each $r\in \{1,2,\ldots,n\}$, denote by $I_r$ the ideal
of $\field{K}[\alpha,\beta]$ generated by the polynomials $\det(P) - 1$,
$\det(Q) -1$ and the coefficients of each $x \in \{ 1 \} \cup X$
in the $(i,j)$ entries of the matrix $P L Q$ for $1 \le i \le r$,
$r \le j \le n$. Then the linear matrix $L$ is full if and only
if for each $r\in \{1,2,\ldots,n\}$, the ideal $I_r$ is trivial.
\end{theorem}

\begin{remark}
Notice that there is a misprint in
\cite{Cohn1999a}
\ and the coefficients of $L_0$ are omitted.
\end{remark}

So far we were not able to apply this theorem practically
for $n \ge 5$, where $50$ or more unknowns are involved.
However, if we have any ALS (or linear representation)
for $f-g$, say from Proposition \ref{pro:wp.ratop},
then we could check whether it can be (admissibly) transformed 
into a smaller system, for example $A's' = 0$.
For polynomials (with $A = I-Q$ and $Q$ upper triangular
and nilpotent)
this could be done row by row.
In general the pivot blocks (the blocks in the diagonal)
can be arbitrarily large.
Therefore this elimination has to be done \emph{blockwise}
by setting up a single linear system for row \emph{and}
column operations.
This idea is used in the following lemma.
Note that the existence of a solution for this linear system
is \emph{invariant} under admissible transformations
(on the subsystems). This is a
key requirement since the normal form
\cite{Cohn1994a}
\ is defined modulo similarity transformations
(more general by stable association, Definition~\ref{def:wp.ass}).

\begin{theorem}[%
\protect{\cite[Theorem 1.4]{Cohn1999a}
}]\label{thr:wp.cohn99.14}
If $\pi' = (u',A',v')$ and $\pi''=(u'',A'',v'')$ are equivalent
(pure) linear representations, of which the first is minimal,
then the second is isomorphic to a representation $(u,A,v)$
which has the block decomposition
\begin{displaymath}
u =
\begin{bmatrix}
* & u' & . 
\end{bmatrix},\quad
A =
\begin{bmatrix}
* & . & . \\
* & A' & . \\
* & * & * 
\end{bmatrix}
\quad\text{and}\quad
v = 
\begin{bmatrix}
. \\ v' \\ *
\end{bmatrix}.
\end{displaymath}
\end{theorem}

\smallskip

\begin{lemma}\label{lem:wp.ri}
Let $f,g \in \freeFLD{\field{K}}{X}$ be given by the
admissible linear systems
$\als{A}_f = (u_f, A_f, v_f)$ and $\als{A}_g = (u_g, A_g, v_g)$
of dimension $n_f$ and $n_g$ respectively.
If there exist matrices $T,U \in \field{K}^{n_f \times n_g}$
such that $u_f U = 0$, $T A_g - A_f U = A_f u_f\trp u_g$ and $T v_g = v_f$,
then $f = g$.
\end{lemma}

\begin{proof}
The difference $f - g$ can be represented by the
admissible linear system $A s = v$ with
\begin{displaymath}
A =
\begin{bmatrix} A_f & -A_f u_f\trp u_g \\ . & A_g \end{bmatrix},
\quad
s = \begin{bmatrix} s_f - u_f\trp g \\ -s_g \end{bmatrix}
\quad\text{and}\quad
v = \begin{bmatrix} v_f \\ -v_g \end{bmatrix}.
\end{displaymath}
Defining the (invertible) transformations
\begin{displaymath}
P = 
\begin{bmatrix}
I_{n_f} & T \\
. & I_{n_g}
\end{bmatrix}
\quad\text{and}\quad
Q =
\begin{bmatrix}
I_{n_f} & -U \\
. & I_{n_g}
\end{bmatrix}
\end{displaymath}
and $A' = P A Q$, $s' = Q\inv s$ and $v' = P v$ we get a
new system $A' s' = v'$:
\begin{align*}
A' &= 
\begin{bmatrix} I_{n_f} & T \\ . & I_{n_g} \end{bmatrix}
\begin{bmatrix} A_f & -A_f u_f\trp u_g \\ . & A_g \end{bmatrix}
\begin{bmatrix} I_{n_f} & -U \\ . & I_{n_g} \end{bmatrix} \\
  &=
\begin{bmatrix} A_f & - A_f u_f\trp u_g + TA_g \\ . & A_g \end{bmatrix}
\begin{bmatrix} I_{n_f} & -U \\ . & I_{n_g} \end{bmatrix} \\
  &=
\begin{bmatrix} A_f & -A_f u_f\trp u_g + T A_g - A_f U \\ . & A_g \end{bmatrix}
  =
\begin{bmatrix} A_f & 0 \\ . & A_g \end{bmatrix},\\
s' &= 
\begin{bmatrix} I_{n_f} & U \\ . & I_{n_g} \end{bmatrix}
\begin{bmatrix} s_f - u_f\trp g \\ -s_g \end{bmatrix}
  =
\begin{bmatrix} s_f - u_f\trp g - U s_g \\ -s_g \end{bmatrix}, \\
v' &=
\begin{bmatrix} I_{n_f} & T \\ . & I_{n_g} \end{bmatrix}
\begin{bmatrix} v_f \\ -v_g \end{bmatrix}
  =
\begin{bmatrix} v_f -T v_g \\ -v_g \end{bmatrix}.
\end{align*}
Invertibility of $A'$ over the free field implies $s_f - u_f\trp g - U s_g = 0$,
in particular 
\begin{align*}
0 &= u_f s_f - u_f u_f\trp g - u_f U s_g \\
  &= f - g
\end{align*}
because $u_f U = 0$.
\end{proof}

Let $d$ be the number of letters in the alphabet $X$,
$\dim \als{A}_f = n_f$ and $\dim \als{A}_g = n_g$.
To determine the transformation matrices $T,U \in \field{K}^{n_f \times n_g}$
from the lemma we just have to solve a linear system of 
$(d+1)n_f(n_g+1)$ equations in $2 n_f n_g$ unknowns.
If there is a solution then $f=g$.
Neither $\als{A}_f$ nor $\als{A}_g$ have to be minimal.
Computer experiments show, that Hua's identity
\cite{Amitsur1966a}
\begin{displaymath}
x - \bigl(x\inv + (y\inv - x)\inv \bigr)\inv = xyx
\end{displaymath}
can be tested positively by Lemma~\ref{lem:wp.ri} when
the left hand side is constructed 
by the rational operations from Proposition \ref{pro:wp.ratop}.
However, without assuming minimality,
the fact that there is no solution does \emph{not} imply,
that $f\neq g$, see Example~\ref{ex:wp.path} below.

\begin{theorem}[``Linear solution'' of the Word Problem]\label{thr:wp.wp}
Let $f,g \in \freeFLD{\field{K}}{X}$ be given by the
\emph{minimal} admissible linear systems
$\als{A}_f = (u_f, A_f, v_f)$ and $\als{A}_g = (u_g, A_g, v_g)$
of dimension $n$ respectively.
Then $f = g$ if and only if there exist matrices $T,U \in \field{K}^{n\times n}$
such that $u_f U = 0$, $T A_g - A_f U= A_f u_f\trp u_g$ and $T v_g = v_f$.
\end{theorem}

\begin{proof}
If $f=g$ then, since admissible linear systems correspond to
(pure) linear representations, by Theorem~\ref{thr:wp.cohn99.14}
there exist invertible matrices
$P, Q \in \field{K}^{n\times n}$ 
such that $A_f = P A_g Q$ and $v_f = P v_g$.
Let $T = P$ and $U = Q\inv - u_f\trp u_g$.
The admissible linear systems are minimal. Hence, the left family
$s_f$
is $\field{K}$-linearly independent.
Since the first component of $s_g$ is equal to the first component
of $s_f = Q\inv s_g$ and the left family $s_g$ is $\field{K}$-linearly
independent, the first row of $Q\inv$
must be $[1,0,\ldots,0]$.
Therefore $u_f U = u_f(Q\inv - u_f\trp u_g) = 0$.
Clearly $v_f = T v_g$ and
\begin{align*}
T A_g - A_f U & = P A_g - A_f Q\inv + A_f u_f\trp u_g \\
  &= A_f u_f\trp u_g.
\end{align*}
The other implication follows from Lemma~\ref{lem:wp.ri}.
\end{proof}

\begin{Example}\label{ex:wp.path}
Let $f=x\inv$ and $g=x\inv$ be given by
the admissible linear systems
\begin{displaymath}
[x] s_f = [1]
\quad\text{and}\quad
\begin{bmatrix}
x & -z \\
. & 1
\end{bmatrix}
s_g
=
\begin{bmatrix}
1 \\ .
\end{bmatrix}
\end{displaymath}
respectively. Then the ALS
\begin{displaymath}
\begin{bmatrix}
x & -x & . \\
. & x & -z \\
. & . & 1
\end{bmatrix}
s =
\begin{bmatrix}
1 \\ -1 \\ .
\end{bmatrix},\quad
s =
\begin{bmatrix}
0 \\ -x\inv \\ 0
\end{bmatrix}
\end{displaymath}
represents $f - g = 0$.
\end{Example}

While it is obvious here that the second component of the
solution vector $s_g$ is zero, it is not clear in general how one can
exclude such ``pathological'' LRs
without minimality assumption.
One might ask,
for which class of constructions (rational operations 
\cite{Cohn1999a}
, Higman's trick\index{Higman's trick}
\cite{Higman1940a,Cohn1985a}
,
selfadjoint linearization trick\index{selfadjoint linearization trick}
\cite{Anderson2013a}
, etc.) there are sufficient conditions for
the existence of matrices $T,U$ (over $\field{K}$) 
in Lemma~\ref{lem:wp.ri} if $f=g$.
Unfortunately this seems to be impossible except for some specific examples
as the following ALS (constructed by the
rational operations from Proposition~\ref{pro:wp.ratop})
for $x - x y y\inv = 0$ suggests (some zeros are kept to emphasize the block structure):
\begin{displaymath}
\begin{bmatrix}
1 & -x & -1 & . & . & . & . & . & . \\
0 & 1 & . & . & . & . & . & . & . \\
. & . & 1 & -x & . & . & . & . & . \\
. & . & 0 &  1 & -1 & . & . & . & . \\
. & . & . & . & 0 & 1 & -y & . & . \\
. & . & . & . & -1 & 0 & 1 & . & . \\
. & . & . & . & 0 & 1 & 0 & -1 & . \\
. & . & . & . & . & . & . & 1 & -y \\
. & . & . & . & . & . & . & 0 & 1 
\end{bmatrix}
s =
\begin{bmatrix}
. \\ 1 \\ . \\ . \\ . \\ . \\ . \\ . \\ -1
\end{bmatrix}.
\end{displaymath}
There are no $T,U$ and therefore no $P,Q$ (admissible, with blocks $T,U$)
such that $PAQ$ has a $2\times 7$ upper
right block of zeros and the first two components of
$Pv$ are zero. Therefore we would like to construct
minimal linear representations directly. For regular
elements, algorithms are known, see Section~\ref{sec:wp.reg}.
How to proceed in general is open except for
the inverse. 
This is discussed in Section~\ref{sec:wp.min}.

\section{Regular Elements}\label{sec:wp.reg}

For regular elements (Definition \ref{def:wp.reg})
in the free field 
minimal linear representations can be obtained 
via the Extended Ho-Algorithm 
\cite{Fornasini1980a}
\ from the Hankel matrix\index{Hankel matrix}
or by minimizing a given linear representation
via the algorithm in \cite{Cardon1980a}
\ by detecting linearly dependent rows in the
controllability matrix and linearly dependent columns in the
observability matrix, see Definition~\ref{def:mtx.co}.
The basic idea goes back to Schützenberger
\cite{Schutzenberger1961b}
.
Controllability and observability is discussed in
\cite[Chapter~10]{Kalman1969a}
.

\smallskip
For an alphabet $X = \{ x_1, x_2, \ldots, x_d \}$ (finite, non-empty set),
the free monoid generated by $X$ is denoted by $X^*$.
A formal power series (in non-commuting variables)
is a mapping $f$ from $X^*$ to a commutative field $\field{K}$,
written as formal sum
\begin{displaymath}
f = \sum_{w\in X^*} (f,w)\, w
\end{displaymath}
with coefficients $(f,w)\in \field{K}$.
In general,
$\field{K}$ could be replaced by a ring or a skew field.
\cite{Salomaa1978a}
\ and
\cite{Berstel2011a}
\ contain detailed introductions.
On the set of formal power series
$\ncPOWS{\field{K}}{X}$ the following
rational operations are defined for $f,g,h \in \ncPOWS{\field{K}}{X}$
with $(h,1) = 0$, 
and $\mu \in \field{K}$:

\smallskip
\noindent
The scalar multiplication
\begin{displaymath}
\mu f = \sum_{w\in X^*} \mu(f,w) \, w.
\end{displaymath}
The sum
\begin{displaymath}
f + g = \sum_{w\in X^*} \bigl( (f,w) + (g,w) \bigr) \, w.
\end{displaymath}
The product
\begin{displaymath}
f g = \sum_{w\in X^*} \left( \sum_{uv=w} (f,u)(g,v) \right) \, w.
\end{displaymath}
And the quasiinverse
\begin{displaymath}
h^+ = \sum_{n\ge 1} h^n.
\end{displaymath}

The set of non-commutative (nc) rational series $\ncRATS{\field{K}}{X}$
is the smallest rationally closed (that is, closed under scalar
multiplication, sum, product and quasiinverse) subset
of $\ncPOWS{\field{K}}{X}$
containing the nc polynomials $\freeALG{\field{K}}{X}$.
A series $f\in \ncPOWS{\field{K}}{X}$ is called \emph{recognizable}
if there exists a natural number $n$, a monoid homomorphism
$\mu : X^* \to \field{K}^{n\times n}$ and two vectors
$\alpha\in\field{K}^{1 \times n}$, $\beta\in\field{K}^{n\times 1}$
such that $f$ can be written as
\begin{displaymath}
f = \sum_{w\in X^*} \alpha\, \mu(w)\, \beta\, w.
\end{displaymath}
The triple $(\alpha, \mu, \beta)$ is called a \emph{linear representation}
\cite[Section~II.2]{Salomaa1978a}
.

\begin{theorem}[\cite{Schutzenberger1961b}
] 
A series $f\in \ncPOWS{\field{K}}{X}$ is rational if and only
if it is recognizable.
\end{theorem}

A rational series $f$ can be represented by a \emph{proper
linear system} (PLS for short) $s = v + Qs$
where $f$ is the
first component of the \emph{unique} solution vector $s$
(with $v \in \field{K}^{n \times 1}$, 
$Q = Q_1 \otimes x_1 + Q_2 \otimes x_2 + \ldots + Q_d \otimes x_d$,
$Q_i \in \field{K}^{n \times n}$ for some $n\in \numN$).
Rational operations are then formulated on this level
\cite[Section~II.1]{Salomaa1978a}
. Clearly, every proper linear system gives rise to an admissible
linear system $\als{A} = (u, I-Q, v)$ with $u = e_1$.
When $(\alpha, \mu, \beta)$ is a linear representation
of a recognizable series, then $\pi = (\alpha, I-Q, \beta)$
with
\begin{equation}\label{eqn:wp.regQ}
Q = \sum_{x\in X} \mu(x) \otimes x.
\end{equation}
is a linear representation of $f \in \freeFLD{\field{K}}{X}$.
For a PLS the solution vector $s$ can be computed by
the quasiinverse $Q\quinv$:
\begin{align*}
s &= (I-Q)\inv v \\
  &= (I+Q\quinv) v \\
  &= (I + Q + Q^2 + Q^3 + \ldots) v.
\end{align*}

\begin{definition}[Controllability and Observability Matrix]\label{def:mtx.co}
\index{controllability matrix}\index{observability matrix}%
Let $\als{P} = (u,I-Q,v)$ be a proper linear system of dimension $n$
(for some nc rational series).
Then the \emph{controllability matrix} and the \emph{observability matrix}
are defined as
\begin{displaymath}
\mathcal{V} =
\mathcal{V}(\als{P}) = \begin{bmatrix} v &  Qv & \ldots & Q^{n-1}v \end{bmatrix}
\quad\text{and}\quad
\mathcal{U} = 
\mathcal{U}(\als{P}) = 
\begin{bmatrix}
u \\ uQ \\ \vdots \\ uQ^{n-1}
\end{bmatrix}
\end{displaymath}

\vspace{-1.5ex}
\noindent
respectively.
\end{definition}

\begin{remark}
Note that the monomials (in the polynomials) in $Q^k$ have length $k$.
The matrices $\mathcal{V}$ and $\mathcal{U}$ 
are over $\freeALG{\field{K}}{X}$.
A priori these matrices would have an infinite number of
columns and rows respectively. However,
by
\cite[Lemma~6.6.3]{Cohn1995a}
, it suffices to use the columns of $\mathcal{V}$ and rows of $\mathcal{U}$
only.
This gives the connection to
\cite[Section I.2]{Berstel2011a}
\ and could be used for minimizing proper linear systems
\cite{Salomaa1978a}
.
In other words: Instead of identifying $\field{K}$-linear dependence
of the left family $s = (I-Q)\inv v = (I + Q + Q^2 + \ldots) v$,
we can restrict to the ``approximated'' left family
$\tilde{s} = (I + Q + \ldots + Q^{n-1})v$.
\end{remark}

Now let $X_{k}^* \subseteq X^*$ denote the set of words of length $k$
and use $\mu:X^* \to \field{K}^{n \times n}$ from~\eqref{eqn:wp.regQ}
to define
\raisebox{0pt}[0pt][0pt]{$V_k \in \field{K}^{n \times d^k}$}
with columns $\mu(w) v$ for $w \in X_k^*$ and
\raisebox{0pt}[0pt][0pt]{$U_k \in \field{K}^{d^k \times n}$}
with rows $u \mu(w)$ for $w \in X_k^*$.
Then the \emph{controllability matrix} and 
the \emph{observability matrix} can be defined alternatively as
\vspace{-1.5ex}
\begin{displaymath}
\mathcal{V}' =
\begin{bmatrix}
V_0 & V_1 & \ldots & V_{n-1}
\end{bmatrix}
\quad\text{and}\quad
\mathcal{U}' =
\begin{bmatrix}
U_0 \\ U_1 \\ \vdots \\ U_{n-1}
\end{bmatrix}
\quad\text{respectively}
\end{displaymath}
with entries in $\field{K}$.
Note that the rank of $\mathcal{V}'$ is at most $n$ while
the number of columns of the blocks $V_k$ is $d^k$.
So ---for an alphabet with more than one letter--- most of
the columns are not needed.
For $X=\{ x \}$ and $Q = Q_x \otimes x$ we have $V_k = Q_x^k v$
and $U_k = u Q_x^k$. Hence $\mathcal{V}'$ and $\mathcal{U}'$ can be written
as
\vspace{-1.5ex}
\begin{displaymath}
\mathcal{V}' =
\begin{bmatrix}
v & Q_x v & \ldots & Q_x^{n-1} v
\end{bmatrix}
\quad\text{and}\quad
\mathcal{U}' =
\begin{bmatrix}
u \\ u Q_x \\ \vdots \\ u Q_x^{n-1}
\end{bmatrix}
\quad\text{respectively.}
\end{displaymath}
Compare with 
\cite[Section~6.3]{Kalman1969a}
. For controllability and observability in connection with realizations
(Definition~\ref{def:wp.realization}) see also 
\cite{Helton2006a}
.


\section{Minimizing the Inverse}\label{sec:wp.min}

Having solved the word problem, our next goal is ``minimal arithmetics''
in the free field $\freeFLD{\field{K}}{X}$. That is, 
given elements by minimal admissible linear systems,
to compute minimal ones for the rational operations.
For the scalar multiplication this is trivial.
For the inverse some preparation is necessary.
The result is presented in Theorem~\ref{thr:wp.mininv}.
The ``minimal sum'' and the ``minimal product''
are considered in future work.
The main difficulty is \emph{not} minimality but
the restriction to \emph{linear} techniques.

\begin{proposition}\label{pro:min.mon}
Let $k \in \numN$ and $f= x_{i_1} x_{i_2} \cdots x_{i_k}$ be a monomial in
$\freeALG{\field{K}}{X} \subseteq \freeFLD{\field{K}}{X}$.
Then
\begin{displaymath}
\als{A} = \left(
\begin{bmatrix}
1 & .  & \cdots & .
\end{bmatrix},
\begin{bmatrix}
1 & -x_{i_1} \\
  & 1 & -x_{i_2} \\
  & & \ddots & \ddots \\
  & & & 1 & -x_{i_k} \\
  & & & & 1
\end{bmatrix},
\begin{bmatrix}
. \\ .  \\ \vdots \\ .  \\ 1
\end{bmatrix}
\right)
\end{displaymath}
is a \emph{minimal} ALS of dimension $\dim \als{A} = k+1$.
\end{proposition}

\begin{proof}
The system matrix of $\als{A}$ is full.
For row indices $[1,x_{i_1},x_{i_1}x_{i_2},$ $\ldots,x_{i_1}\cdots x_{i_k}]$ and
column indices $[1,x_{i_k},x_{i_{k-1}}x_{i_k},$ $\ldots,x_{i_1}\cdots x_{i_k}]$
the Hankel matrix 
\cite{Fliess1974a}
,
\cite[Section~II.3]{Salomaa1978a}
\ of $f$ is
\begin{displaymath}
H(f) = 
\begin{bmatrix}
& & & 1 \\
& & 1 \\
& \revddots \\
1
\end{bmatrix}
\end{displaymath}
and has rank $k+1$.
\end{proof}

\begin{remark}
Trivially, $\als{A} = ([1], [1], [1])$ is a minimal ALS 
for the \emph{unit element} (empty word).
\end{remark}

The following proposition is a variant of the inverse in
Proposition~\ref{pro:wp.ratop} and is motivated by 
inverting the inverse of a monomial, for example,
$f = (xyz)\inv$. A minimal ALS for $f$
is given by
\begin{displaymath}
\begin{bmatrix}
z & -1 & . \\
. & y & -1 \\
. & . & x
\end{bmatrix}
s =
\begin{bmatrix}
. \\ . \\ 1
\end{bmatrix},
\quad
s =
\begin{bmatrix}
z\inv y\inv x\inv \\
y\inv x\inv \\
x\inv
\end{bmatrix}.
\end{displaymath}
Minimality is clear immediately by also checking the $\field{K}$-linear
independence of the right family. Using the construction of the inverse
from Proposition~\ref{pro:wp.ratop} we get the system
\begin{displaymath}
\begin{bmatrix}
. & z & -1 & . \\
. & . & y & -1 \\
-1 & . & . & x \\
. & 1 & . & . 
\end{bmatrix}
s =
\begin{bmatrix}
. \\ . \\ . \\ 1
\end{bmatrix},
\quad
s =
\begin{bmatrix}
xyz \\
1 \\
z \\
yz
\end{bmatrix}.
\end{displaymath}
for $f\inv = xyz$. To obtain the form of Proposition~\ref{pro:min.mon}
we have to reverse the rows $1,2,3$ and the columns $2,3,4$
and multiply the rows $1,2,3$ by $-1$.

\begin{proposition}[Standard Inverse]\label{pro:inv.std}
Let $0 \neq f \in \freeFLD{\field{K}}{X}$ 
be given by the admissible linear system $\als{A} = (u,A,v)$
of dimension $n$.
Then an admissible linear system of dimension $n+1$
for $f\inv$ is given by
\begin{equation}\label{eqn:wp.inv0}
\als{A}\inv = \left(
\begin{bmatrix}
1 & . \\
\end{bmatrix},
\begin{bmatrix}
\perm v & -\perm A \perm \\
. & u \perm
\end{bmatrix},
\begin{bmatrix}
. \\ 1
\end{bmatrix}
\right).
\end{equation}
(Recall that the permutation matrix $\perm = \perm_n$ reverses the order of rows/columns.)
\end{proposition}

\begin{definition}[Standard Inverse]
\index{standard inverse}
Let $\als{A}$ be an ALS for a non-zero element.
We call the ALS~\eqref{eqn:wp.inv0} the \emph{standard inverse} of $\als{A}$,
denoted by $\als{A}^{-1}$.
\end{definition}

\begin{proof}
The reader can easily verify that the solution vector of $\als{A}\inv$ is
\begin{displaymath}
\begin{bmatrix}
f\inv \\ \perm_n s_f f\inv
\end{bmatrix}.
\end{displaymath}
Compare with Proposition~\ref{pro:wp.ratop}.
\end{proof}

We proceed with the calculation of minimal admissible linear systems
for the inverse. We distinguish four types of ALS according to the
form of the system matrix.
Later, in the remark following Lemma~\ref{lem:wp.for1} and~\ref{lem:wp.for2},
we will see how to bring a system matrix to one of these forms depending
on the left and right families.

\begin{lemma}[Inverse Type~$(1,1)$]\label{lem:wp.inv3}
Assume that $0 \neq f \in \freeFLD{\field{K}}{X}$
has a \emph{minimal} admissible linear system of dimension $n$ of the
form
\begin{equation}\label{eqn:wp.inv3a}
\als{A} = \left(
\begin{bmatrix}
1 & . & .
\end{bmatrix},
\begin{bmatrix}
1 & b' & b \\
. & B & b'' \\
. & . & 1
\end{bmatrix},
\begin{bmatrix}
. \\ . \\ \lambda
\end{bmatrix}
\right).
\end{equation}
Then a minimal ALS for $f\inv$ 
of dimension $n-1$ is given by
\begin{equation}\label{eqn:wp.inv3b}
\als{A}'=\left(
\begin{bmatrix}
1 & . 
\end{bmatrix},
\begin{bmatrix}
-\lambda \perm b'' & -\perm B \perm \\
-\lambda b & - b'\perm
\end{bmatrix},
\begin{bmatrix}
. \\ 1
\end{bmatrix}
\right)
\end{equation}
with $1\notin R(\als{A}')$ and $1\notin L(\als{A}')$.
\end{lemma}

\begin{proof}
The standard inverse of $\als{A}$ is
\begin{displaymath}
\begin{bmatrix}
\lambda & -1 & . & . \\
. & -\perm b'' & -\perm B \perm & . \\
. & -b & -b'\perm & -1 \\
. & . & . & 1 
\end{bmatrix}
\tilde{s} = 
\begin{bmatrix}
. \\ . \\ . \\ 1
\end{bmatrix}.
\end{displaymath}
Adding row 4 to row 3 and $\lambda$-times column 2 to column 1 gives
\begin{displaymath}
\begin{bmatrix}
0 & -1 & . & . \\
-\lambda \perm b'' & -\perm b'' & -\perm B \perm & . \\
-\lambda b & -b & -b'\perm & 0 \\
. & . & . & 1 
\end{bmatrix}
s' = 
\begin{bmatrix}
. \\ . \\ 1 \\ 1
\end{bmatrix}.
\end{displaymath}
It follows that $s_2'=0$ and $s_{n+1}'$ does not contribute to the
solution $s_1'$ and thus the first and the last row as well as
the second and the last column can be removed.
If $\als{A}'$ were not minimal, then there
would exist a system $\als{A}''$ of dimension $m < n-1$ for $f\inv$.
The standard inverse $(\als{A}'')\inv$ would
give a system of dimension $m+1 < n$ for $f$,
contradicting minimality of $\als{A}$.
It remains to show that $1\notin R(\als{A}')$ and
$1\notin L(\als{A}')$.
Let $t = (t_1, t_2, \ldots, t_n)$ be the right family of $\als{A}$
which is (due to minimality) $\field{K}$-linearly independent.
Then the right family of $\als{A}\inv$ is
$(f\inv t_n, \ldots, f\inv t_2, f\inv t_1, f\inv)$,
that after the row operation becomes $f\inv (t_n, \ldots, t_2, t_1, 1-t_1)$.
Removing the first and the last component (corresponding to
the first and the last row) yields the right family
$f\inv (t_{n-1}, \ldots, t_2, t_1)$. Therefore $1\notin R(\als{A}')$,
because otherwise $f \in \linsp \{ t_{n-1}, \ldots, t_2, t_1 \}$,
contradicting $\field{K}$-linear independence of $t$.
Similar arguments show that $1\notin L(\als{A}')$.
\end{proof}

\begin{lemma}[Inverse Type~$(1,0)$]\label{lem:wp.inv2}
Assume that
$0 \neq f \in \freeFLD{\field{K}}{X}$
has a \emph{minimal} admissible linear system of dimension $n$
of the form
\begin{equation}\label{eqn:wp.inv2a}
\als{A} = \left(
\begin{bmatrix}
1 & . & .
\end{bmatrix},
\begin{bmatrix}
1 & b' & b \\
. & B & b'' \\
. & c' & c
\end{bmatrix},
\begin{bmatrix}
. \\ . \\ \lambda
\end{bmatrix}
\right)
\end{equation}
with $1 \not\in L(\als{A})$.
Then a minimal ALS for $f\inv$
of dimension $n$ is given by
\begin{equation}\label{eqn:wp.inv2b}
\als{A}' = \left(
\begin{bmatrix}
1 & . & .
\end{bmatrix},
\begin{bmatrix}
1 & - \frac{1}{\lambda} c & - \frac{1}{\lambda} c' \perm \\
. & -\perm b'' & -\perm B \perm \\
. & -b & -b'\perm 
\end{bmatrix},
\begin{bmatrix}
. \\ . \\ 1
\end{bmatrix}
\right)
\end{equation}
with $1\notin L(\als{A}')$.
\end{lemma}

\begin{proof}
The standard inverse of $\als{A}$ is
\begin{displaymath}
\begin{bmatrix}
\lambda & -c & -c' \perm & . \\
. & -\perm b'' & -\perm B \perm & . \\
. & -b & -b'\perm & -1 \\
. & . & . & 1 
\end{bmatrix}
\tilde{s} = 
\begin{bmatrix}
. \\ . \\ . \\ 1
\end{bmatrix}
\end{displaymath}
and has dimension $n+1$.
After adding row~$n+1$ to row~$n$ we can remove row and column~$n+1$
because $\tilde{s}_{n+1}$ does not contribute to the solution $\tilde{s}_1 = f\inv$.
Then we divide the first row by $\lambda$
and obtain \eqref{eqn:wp.inv2b}.
It remains to show that $\als{A}'$ is minimal
and $1\notin L(\als{A}')$.
Let $(s_1, s_2, \ldots, s_n)$ be the left family of $\als{A}$
which is (due to minimality) $\field{K}$-linearly independent.
Then the left family of $\als{A}\inv$ is
$(f\inv, s_n f\inv, \ldots, s_2 f\inv, 1)$.
Note that (admissible) row operations do not affect the
left family.
Since we removed the last entry $s_1 f\inv = \tilde{s}_{n+1}=1$, the left family
of $\als{A}'$ is
$(1, s_n, \ldots, s_2) f\inv$. By assumption $1 \notin L(\als{A})$.
Therefore $1 \notin \linsp \{ s_2, s_3, \ldots, s_n \}$,
hence $(1,s_n, \ldots, s_2)$ is $\field{K}$-linearly independent.
Clearly, $1\notin L(\als{A}')$ because
$f \notin \linsp \{ 1, s_n, \ldots, s_2 \}$.
Similarly, let $(t_1, t_2, \ldots, t_n)$ be the right family
of $\als{A}$ which is $\field{K}$-linearly independent.
Then the right family of $\als{A}\inv$ is
$(f\inv t_n, \ldots, f\inv t_2, f\inv t_1, f\inv)$,
that after the row operation is
$(f\inv t_n, \ldots, f\inv t_2, f\inv t_1, f\inv -f\inv t_1 )$.
Since we removed the last entry, the right family of $\als{A}'$
is $f\inv (t_n, \ldots, t_2, t_1)$ which is clearly
$\field{K}$-linearly independent.
Proposition~\ref{pro:cohn94.47} gives minimality of $\als{A}'$.
\end{proof}

\begin{lemma}[Inverse Type~$(0,1)$]\label{lem:wp.inv1}
Assume that $0 \neq f \in \freeFLD{\field{K}}{X}$ 
has a \emph{minimal} admissible linear system of dimension $n$
of the form
\begin{equation}\label{eqn:wp.inv1a}
\als{A} = \left(
\begin{bmatrix}
1 & . & . 
\end{bmatrix},
\begin{bmatrix}
a & b' & b \\
a' & B & b'' \\
. & . & 1
\end{bmatrix},
\begin{bmatrix}
. \\ . \\ \lambda
\end{bmatrix}
\right)
\end{equation}
with $1 \not\in R(\als{A})$.
Then a minimal ALS for $f\inv$
of dimension $n$ is given by
\begin{equation}\label{eqn:wp.inv1b}
\als{A}' = \left(
\begin{bmatrix}
1 & . & . 
\end{bmatrix},
\begin{bmatrix}
-\lambda \perm b'' & -\perm B \perm & -\perm a'. \\
-\lambda b & -b'\perm & -a \\
. & . & 1 
\end{bmatrix},
\begin{bmatrix}
. \\ . \\ 1
\end{bmatrix}
\right)
\end{equation}
with $1\notin R(\als{A'})$.
\end{lemma}

\begin{proof}
The standard inverse of $\als{A}$ is
\begin{displaymath}
\begin{bmatrix}
\lambda & -1 & . & . \\
. & -\perm b'' & -\perm B \perm  & -\perm a' \\
. & -b & -b'\perm & -a \\
. & . & . & 1 
\end{bmatrix}
\tilde{s} = 
\begin{bmatrix}
. \\ . \\ . \\ 1
\end{bmatrix}.
\end{displaymath}
After adding $\lambda$-times column~2 to column~1 we can
remove row~1 and column~2, because $\tilde{s}_2=0$.
Showing minimality and $1\notin R(\als{A}')$ is similar
to the proof of Lemma~\ref{lem:wp.inv2}
(column operations affect the left family).
\end{proof}

\begin{lemma}[Inverse Type~$(0,0)$]\label{lem:wp.inv0}
Let $\als{A} = (u,A,v)$ be a minimal admissible linear system
of dimension $n$ for $0 \neq f \in \freeFLD{\field{K}}{X}$
with $1\notin R(\als{A})$ and $1\notin L(\als{A})$.
Then the standard inverse $\als{A}\inv$ is a minimal ALS
of dimension $n+1$ for $f\inv$.
\end{lemma}

\begin{proof}
If $\als{A}\inv$ were not minimal, then there
would be a system $\als{A}'$ of dimension $m < n+1$ for $f\inv$.
Applying Lemma~\ref{lem:wp.inv3} (Inverse Type~$(1,1)$), we would get
a system of dimension $m-1 < n$ for $f$,
contradicting minimality of $\als{A}$.
\end{proof}

\begin{Example}
Taking the minimal ALS (for the anticommutator) from Example~\ref{ex:anticomm},
we get, by Lemma~\ref{lem:wp.inv3} (Inverse Type~$(1,1)$),
a \emph{minimal} ALS for $(xy+yx)\inv$:
\begin{displaymath}
\als{A}'=
\left(
\begin{bmatrix}
1 & . & . 
\end{bmatrix},
\begin{bmatrix}
x & - 1 & . \\
y & . & -1 \\
. & y & x 
\end{bmatrix},
\begin{bmatrix}
. \\ . \\ 1
\end{bmatrix}
\right).
\end{displaymath}
Lemma~\ref{lem:wp.inv0} (Inverse Type~$(0,0)$) gives again a minimal system
for $xy+yx$.
\end{Example}

Since it is possible to construct minimal linear representations
for regular elements (see Section~\ref{sec:wp.reg}), this is in particular
true for polynomials.
These are of the form \eqref{eqn:wp.inv3a},
since $A = I-Q$ with nilpotent $Q$,
which can be choosen upper triangular. This can be seen either by
looking at proper linear systems (Section~\ref{sec:wp.reg})
where admissible transformations
are conjugations (of the system matrix such that the first component
of the solution vector is left untouched) or by the following proposition.

\begin{proposition}[%
\protect{\cite[Proposition 2.1]{Cohn1999a}
}]\label{thr:wp.cohn99.21}
Let $f\in \freeFLD{\field{K}}{X}$.

(i) $f$ is a power series if and only if
in any minimal representation, the constant term of its system matrix
is invertible. There is then a minimal representation which is unital,
that is, $A_0 = I$.

(ii) $f$ is a polynomial if and only if
in any unital minimal representation, the matrix 
$A = A_1 \otimes x_1 + \ldots + A_d \otimes x_d$ is nilpotent.
There is then a minimal representation with a unitriangular
(ones on and zeros below the diagonal) system matrix.
\end{proposition}

\begin{Example}
The element $xyz\inv$ admits the following \emph{minimal} ALS
\begin{displaymath}
\als{A} = \left(
\begin{bmatrix}
1 & . & . 
\end{bmatrix},
\begin{bmatrix}
1 & -x & . \\
. & 1 & -y \\
. & . & z
\end{bmatrix},
\begin{bmatrix}
. \\ . \\ 1
\end{bmatrix}
\right)
\end{displaymath}
with right family $t = (1, x, xyz\inv)$
and left family $s = (xyz\inv, yz\inv, z\inv)$. Now
Lemma~\ref{lem:wp.inv2} (Inverse Type~$(1,0)$) can be applied to get
the minimal ALS
\begin{displaymath}
\als{A}' = \left(
\begin{bmatrix}
1 & . & . 
\end{bmatrix},
\begin{bmatrix}
1 & -z & . \\
. & y & -1 \\
. & . & x
\end{bmatrix},
\begin{bmatrix}
. \\ . \\ 1
\end{bmatrix}
\right)
\end{displaymath}
for $zy\inv x\inv$. Note that we can apply Lemma~\ref{lem:wp.inv2}
again, because $1\notin L(\als{A}')$.
\end{Example}

If an admissible linear system $\als{A}$ is of the form \eqref{eqn:wp.inv1a}
in Lemma~\ref{lem:wp.inv1} (Inverse Type~$(0,1)$), then
it follows immediately that $1 \in L(\als{A})$. Conversely,
assuming $1\in L(\als{A})$,
the proof of the existence of an admissible transformation $(P,Q)$
such that $P\als{A}Q$ is of the form \eqref{eqn:wp.inv1a} 
is a bit more involved and requires minimality.

\begin{lemma}[for Inverse Type~$(0,1)$]\label{lem:wp.for1}
Let $\als{A} = (u,A,v)$ be a \emph{minimal} admissible linear system
with $\dim \als{A} = n \ge 2$ and $1\in L(\als{A})$. Then there
exists an admissible transformation $(P,Q)$ such that
$(uQ, PAQ, Pv)$ is of the form \eqref{eqn:wp.inv1a}.
\end{lemma}

\begin{proof}
Without loss of generality, assume that $v = [0,\ldots,0,1]\trp$ and
the left family $s = A\inv v$ is $(s_1, s_2,$ $\ldots, s_{n-1}, 1)$.
Otherwise it can be brought to this form by some admissible transformation
$(P^\circ, Q^\circ)$.
Now let $\bar{A}$ denote the upper left $(n-1)\times (n-1)$ block of $A$,
let $\bar{s} = (s_1, \ldots, s_{n-1})$ and write $As = v$ as
\begin{displaymath}
\begin{bmatrix}
\bar{A} & b \\
c & d
\end{bmatrix}
\begin{bmatrix}
\bar{s} \\ 1
\end{bmatrix}
=
\begin{bmatrix}
0 \\ 1
\end{bmatrix}.
\end{displaymath}
This system is equivalent to
\begin{displaymath}
\begin{bmatrix}
\bar{A} & b \\
c & d-1
\end{bmatrix}
\begin{bmatrix}
\bar{s} \\ 1
\end{bmatrix}
=
\begin{bmatrix}
0 \\ 0
\end{bmatrix}.
\end{displaymath}
Since the left family is $\field{K}$-linearly independent (by minimality of $\als{A}$),
the matrix
$\tilde{A} = \bigl[\begin{smallmatrix} \bar{A} & b \\ c & d-1 \end{smallmatrix}\bigr]$
cannot be full. We claim that there is only one possibility to transform $\tilde{A}$
to a hollow matrix, namely with zero last row. If we cannot produce a
$(n-i) \times i$ block of zeros (by invertible transformations) in the
first $n-1$ rows of $\tilde{A}$, then we cannot get 
blocks of zeros of size $(n-i+1) \times i$ and we are done.

Now assume that there are invertible matrices
$P' \in \field{K}^{(n-1) \times (n-1)}$
and (admissible) $Q\in\field{K}^{n \times n}$
with $(Q\inv s)_1 = s_1$,
such that $P'[ \bar{A} , b]Q$ contains a zero block of size $(n-i) \times i$
for some $i=1,\ldots, n-1$. There are two cases. If the first $n-i$
entries in column~1 cannot be made zero, we construct an upper right
zero block:
\begin{displaymath}
\hat{A} = 
\begin{bmatrix}
A_{11} & . \\
A_{21} & A_{22}
\end{bmatrix},
\quad \hat{s} = Q\inv s
\quad\text{and}\quad \hat{v} = P v = v
\end{displaymath}
where $A_{11}$ has size $(n-i) \times (n-i)$.
If $A_{11}$ were not full,
then $A$ would not be full (the last row is not involved
in the transformation). Hence this pivot block is
invertible over the free field. 
Therefore $\hat{s}_1 = \hat{s}_2 = \ldots = \hat{s}_{n-i} = 0$.
Otherwise we construct an upper left zero block in $PAQ$.
But then 
$\hat{s}_{i+1} = \hat{s}_{i+2} = \ldots = \hat{s}_{n} = 0$.
Both contradict $\field{K}$-linear independence of the left family.

So there is only one block left, which can make $\tilde{A}$ non-full.
Hence, by Lemma~\ref{lem:cohn95.636},
the modified (system) matrix
$\tilde{A}$
is associated over $\field{K}$ to a linear hollow matrix with
a $1 \times n$ block of zeros, say in the last row (the columns
and the first $n-1$ rows are left unouched):
\begin{displaymath}
\begin{bmatrix}
I_{n-1} & . \\
T & 1
\end{bmatrix}
\begin{bmatrix}
\bar{A} & b \\
c & d-1 
\end{bmatrix}
I_n
=
\begin{bmatrix}
\bar{A} & b \\
0 & 0 
\end{bmatrix}.
\end{displaymath}
Hence we have $T\bar{A} + c = 0$ and therefore
\begin{displaymath}
\begin{bmatrix}
\bar{A} & b \\
0 & Tb + d 
\end{bmatrix}
\begin{bmatrix}
\bar{s} \\ 1
\end{bmatrix}
=
\begin{bmatrix}
0 \\ 1
\end{bmatrix}.
\end{displaymath}
Clearly $Tb + d = 1$. The transformation
\begin{displaymath}
(P,Q) = \left(
\begin{bmatrix}
I_{n-1} & . \\
T & 1 
\end{bmatrix}
P^\circ, Q^\circ \right)
\end{displaymath}
does the job.
\end{proof}

\begin{lemma}[for Inverse Type~$(1,0)$]\label{lem:wp.for2}
Let $\als{A} = (u,A,v)$ be a \emph{minimal} admissible linear system
with $\dim \als{A} = n \ge 2$ and $1\in R(\als{A})$. Then there
exists an admissible transformation $(P,Q)$ such that
$(uQ, PAQ, Pv)$ is of the form \eqref{eqn:wp.inv2a}.
\end{lemma}

\begin{proof}
The proof is similar to the previous one switching the role
of left and right family.
\end{proof}

\begin{remark}
If $1 \in R(\als{A})$ for some \emph{minimal} ALS $\als{A} = (u,A,v)$,
say $\dim \als{A} = n$,
then, by Lemma~\ref{lem:wp.for2}, there is an
admissible transformation $(P,Q)$ such that the first column
in $PAQ$ is $[1,0,\ldots,0]\trp$.
So, if the first column of $A=(a_{ij})$ is not in this form, an
admissible transformation can be found in two steps:
Firstly, we can set up a linear system to determine an $(n-1)$-duple
of scalars $(\mu_2, \mu_3, \ldots, \mu_n)$ such that
$a_{i1} + \mu_2 a_{i2} + \mu_3 a_{i3} + \ldots + \mu_n a_{in}$
is in $\field{K}$ for $i=1,2,\ldots,n$.
Secondly, we use elementary row transformations (Gaussian elimination
in the first column) and ---if necessary--- permutations to get
the desired form of the first column.
Together, these transformations give some (admissible) transformation
$(P',Q')$.

An analogous procedure can be applied if $1\in L(\als{A})$.
And it can be combined for the case as in Lemma~\ref{lem:wp.inv3} (Inverse Type~$(1,1)$).
It works more generally for non-minimal systems, but can
fail in ``pathological'' cases. Compare with Example~\ref{ex:wp.path}.
\end{remark}

\begin{theorem}[Minimal Inverse]\label{thr:wp.mininv}
Let $0 \neq f \in \freeFLD{\field{K}}{X}$ be given by
the minimal system $\als{A} = (u, A, v)$
of dimension $n$.
Then a minimal admissible linear system for $f\inv$ is given by
\begin{displaymath}
\als{A}' = 
\begin{cases}
\text{\eqref{eqn:wp.inv3b} of $\dim \als{A}' = n-1$}
  & \text{if $1\in L(\als{A})$ and $1\in R(\als{A})$,}\\
\text{\eqref{eqn:wp.inv2b} of $\dim \als{A}' = n$}
  & \text{if $1\not\in L(\als{A})$ and $1\in R(\als{A})$,}\\
\text{\eqref{eqn:wp.inv1b} of $\dim \als{A}' = n$}
  & \text{if $1\in L(\als{A})$ and $1\not\in R(\als{A})$ and}\\
\text{\eqref{eqn:wp.inv0} of $\dim \als{A}' = n+1$}
  & \text{if $1\not\in L(\als{A})$ and $1\not\in R(\als{A})$}
\end{cases}
\end{displaymath}
provided that the necessary transformations according to
Lemma~\ref{lem:wp.for1} and~\ref{lem:wp.for2} are done before.
\end{theorem}

\begin{proof}
See
Lemma~\ref{lem:wp.inv3}, \ref{lem:wp.inv2}, \ref{lem:wp.inv1}
and~\ref{lem:wp.inv0}.
\end{proof}

There are two immediate consequences: The first
follows from Proposition~\ref{thr:wp.cohn99.21},
that is, the inverse type $(1,1)$ applies in particular
to polynomials. The second can be used to distinguish
between ``trivial'' units (in the ground field $\field{K}$)
and ``non-trivial'' units, that is, elements in
$\freeFLD{\field{K}}{X} \setminus \field{K}$.

\begin{corollary}
Let $p \in \freeALG{\field{K}}{X}$ with $\rank p = n \ge 2$.
Then $\rank(p\inv) = n-1$.
\end{corollary}

\begin{corollary}\label{cor:wp.rank1}
Let $0 \neq f \in \field{F}$.
Then $f \in \field{K}$ if and only if $\rank(f) = \rank(f\inv) = 1$.
\end{corollary}

\section*{Acknowledgement}

Special thanks go to Marek Bo\.zejko and Victor Vinnikov
for the hospitality and all the motivating discussions
in Wrocław and Beer-Sheva, respectively.
However, I am very grateful to Franz Lehner.
He gave me the chance to enter the world of non-commutative
mathematics. Without the freedom, support and advice
he offered this work would not have been possible.
Additonally I thank the anonymous referees for the
valuable comments.

\ifJOURNAL
\bibliographystyle{plain}
\else
\bibliographystyle{alpha}
\fi
\bibliography{doku}

\end{document}